\documentclass[12pt]{amsart}

\usepackage[
  margin=1in,
  includefoot,
  footskip=30pt,
]{geometry}

\usepackage{amsfonts}
\usepackage{mathrsfs}
\usepackage{amssymb}

\usepackage{amsthm}

\usepackage{enumitem}

\usepackage{longtable}

\usepackage{blkarray}

\usepackage{tikz-cd}

\usepackage{graphicx}

\usepackage{hyperref}
\usepackage{cleveref}

\theoremstyle{plain}
\newtheorem{theorem}{Theorem}[section]
\newtheorem{proposition}[theorem]{Proposition}
\newtheorem{lemma}[theorem]{Lemma}
\newtheorem{corollary}[theorem]{Corollary}

\theoremstyle{definition}
\newtheorem{definition}[theorem]{Definition}
\newtheorem{example}[theorem]{Example}

\def\op{\operatorname}
\renewcommand{\tilde}{\widetilde}

\def\A{\mathbf{A}}
\def\B{\underline{B}}
\def\P{\mathbf{P}}
\def\M{\mathbf{M}}
\def\E{\mathscr{E}}
\def\F{\mathscr{F}}

\def\L{\mathscr{L}}
\def\K{\mathscr{K}}
\def\O{\mathscr{O}}

\def\Spec{\op{Spec}}

\def\a{\underline{a}}
\def\aa{\underline{\alpha}}
\def\b{\underline{b}}
\def\bb{\underline{\beta}}
\def\c{\underline{c}}

\def\d{\underline{d}}

\def\Hom{\op{Hom}}
\def\Aut{\op{Aut}}

\def\sExt{\mathscr{E}\kern -.5pt xt}
\newcommand*{\sHom}{\mathscr{H}\kern -.5pt om}

\def\Betti{\mathcal{B}\kern -.5pt etti}
\def\VB{\mathcal{VB}_{\P^n}^\dagger}

\def\coker{\op{coker}}

\def\rank{\op{rank}}
\def\depth{\op{depth}}
\def\codim{\op{codim}}
\def\reg{\op{reg}}
\def\pdim{\op{pdim}}

\begin{document}

\title{Bundles on $\P^n$ with vanishing lower cohomologies}
\author{Mengyuan Zhang}
\address{Department of Mathematics, University of California, 
Berkeley, CA 94720}
\email{myzhang@berkeley.edu}

\begin{abstract}
We study bundles on projective spaces that have vanishing lower cohomologies using their short minimal free resolutions.
We partition the moduli $\M$ according to the Hilbert function $H$ and classify all possible Hilbert functions $H$ of such bundles. 
For each $H$, we describe a stratification of $\M_H$ by quotients of rational varieties.
We show that the closed strata form a graded lattice. 
\end{abstract}

\maketitle 

\section*{Introduction}
In this paper, we study bundles (i.e. locally free coherent sheaves) $\E$ on $\P^n$ such that
\begin{align*}
H^i(\E(t)) = 0,\quad  \forall t\in\mathbb{Z}, \quad \forall 1\le  i \le n-2. \tag{$\dagger$} \label{Cohomology}
\end{align*}
These include all bundles on $\P^2$ in particular.
Note that (\ref{Cohomology}) is an open condition on a family of bundles by the semicontinuity of cohomologies.
A rank $r$ bundle $\E$ on $\P^n$ satisfying (\ref{Cohomology}) admits a resolution by direct sums of line bundles of the form
\[
0 \to \bigoplus_{i = 1}^{l}\O_{\P^n}(-a_i) \xrightarrow{\varphi} \bigoplus_{i = 1}^{l+r}\O_{\P^n}(-b_i) \to \E \to 0.
\]
A minimal such resolution is unique up to isomorphism, and the integers $\a = (a_1,\dots, a_l)$ and $\b = (b_1,\dots, b_{l+r})$ are invariants of $\E$ called the Betti numbers.

\bigskip

The main results in this paper are the following.

\bigskip

In \Cref{FreeRes}, we classify all Betti numbers of rank $r$ bundles on $\P^n$ satisfying (\ref{Cohomology}), generalizing results from Bohnhorst and Spindler \cite{BS} for the case $r = n$.
Accordingly, we classify all possible Hilbert functions of such bundles, and introduce a compact way to represent and to generate them.
We show that there are only finitely many possible Betti numbers of bundles satisfying (\ref{Cohomology}) with fixed first Chern class and bounded regularity, generalizing the observation of Dionisi and Maggesi \cite{DM} for $r = n = 2$.
We then give examples to show that the semistability of such a bundle is not determined by its Betti numbers in general, in contrast to the case when $r = n$ discussed in \cite{BS}.

In \Cref{Stratification}, we define natural topologies on $\VB(H)$ and $\VB(\a,\b)$, the set of isomorphism classes of bundles on $\P^n$ satisfying (\ref{Cohomology}) with Hilbert function $H$ and with Betti numbers $(\a,\b)$ respectively.
The topologies are induced from the rational varieties of matrices whose ideals of maximal minors have maximal depth.
We show that all Betti numbers of bundles in $\VB(H)$ form a graded lattice under the partial order of canceling common terms.
This lattice is downward closed and infinite in general, where the subposet of Betti numbers up to any given regularity is a finite graded sublattice.
Finally, we describe the stratification of $\VB(H)$ by various subspaces $\VB(\a,\b)$.
We show that the closed strata intersect along another closed stratum, and that they form a graded lattice dual to the lattice of Betti numbers.
An open subset $\VB(H)^{ss}$ of $\VB(H)$ is a subscheme of the coarse moduli space $\mathcal{M}(\chi)$ of semistable torsion-free coherent sheaves with Hilbert polynomial $\chi$, similarly for an open subset $\VB(\a,\b)^{ss}$ of $\VB(\a,\b)$. 
The same description applies to the stratification of $\VB(H)^{ss}$ by $\VB(\a,\b)^{ss}$ on the level of topological spaces.

\bigskip
The study of vector bundles on algebraic varieties is central to algebraic geometry. 
In particular, the study of bundles on projective spaces already presents interesting challenges.
We do not attempt to give a survey of the subject here.
Instead, we provide some historical perspectives to motivate the investigations in this paper.

\bigskip


Maruyama \cite{Maruyama1} proved that the coarse moduli space of rank two semistable bundles on a smooth projective surface exists as a quasi-projective scheme.
In the same paper, it was shown that the coarse moduli space $\mathcal{M}_{\P^2}(2,c_1,c_2)^s$ of stable rank two bunldes on $\P^2$ with given Chern classes is smooth and irreducible.
Following this development, Barth \cite{Barth} showed that $\mathcal{M}_{\P^2}(2,c_1,c_2)^s$ is connected and rational for $c_1$ even, and Hulek \cite{Hulek} did the same for $c_1$ odd.
Their arguments contained a gap which was pointed out and partially fixed in \cite{Gap} and independently in \cite{ES}.
The existence of the coarse moduli space of semistable torsion-free sheaves of arbitrary rank on a smooth projective variety was finally established by Maruyama \cite{Maruyama}, see \cite{Maruyama2} for another exposition.

Despite the progress in the theory of moduli, there are many basic questions about bundles on projective spaces that are unanswered, see \cite{Problem} for a problem list.
For example, Hartshorne's conjecture \cite{Hartshorne} states that a rank $r$ bundle on $\P^n$ is the direct sum of line bundles when $r<\frac{1}{3}n$.
In particular, the conjecture predicts that rank two bundles on $\P^n$ are split when $n\ge 7$. 
On the other hand, the only known example (up to twists and finite pullbacks) of an indecomposable rank two bundle on $\P^4$ and above is the Horrocks-Mumford bundle \cite{Horrocks}.
It is fair to say that bundles of small rank on $\P^n$ remain mysterious.

\bigskip

This paper is motivated by two main objectives in the study of bundles.

(1) To classify certain invariants of bundles on $\P^n$.

To expand on this point, the Hilbert polynomial is an important invariant of a bundle that is constant in a connected flat family, and thus indexes the connected components of the moduli space (of some subclasses of bundles, e.g. semistable).
Note that the Hilbert polynomial can be computed from the Chern polynomial and vice versa.
Thus the classification of Hilbert polynomials is equivalent to the classification of Chern classes.
The Hilbert function eventually agrees with the Hilbert polynomial, and thus provides finer information.
Furthermore, the Hilbert function can be computed from the Betti numbers of a free resolution of (the section module of) a bundle.
Therefore the Betti numbers are even finer invariants of a bundle.
Consequently, the classification of Betti numbers of bundles will lead to a classification of the Hilbert functions and Hilbert polynomials (equivalently Chern classes).
In this paper, we take the first step by classifying the Betti numbers of bundles when their resolutions are short.
It turned out that this condition implies that these bundles have rank greater than the dimension of the ambient projective space with the exceptions of direct sums of line bundles.

(2) To provide examples of bundles with given invariants.

In the best scenario, the moduli space or the space of isomorphism classes of  bundles with given invariants is unirational, in which case the image of a random point in the projective space will give us a ``random" bundle with given invariants. 
For example, Barth's parametrization of $\mathcal{M}_{\P^2}(c_1,c_2)^{s}$ using nets of rank two quadrics \cite{Barth} allows us to produce ``random" rank two bundles on $\P^2$ with given Chern classes.
Here we can see the importance of using finer invariants.
Since $\M_{\P^2}(c_1,c_2)^{s}$ is irreducible, a general bundle produced using Barth's parametrization will be presented by a matrix of linear and quadratic polynomials by the main theorem in \cite{DM}.
Therefore producing a bundle that is presented by a matrix of forms of other degrees, which is special in the moduli $\mathcal{M}_{\P^2}(c_1,c_2)^{s}$, is like looking for a needle in a haystack.
On the other hand, if we stratify $\mathcal{M}_{\P^2}(c_1,c_2)^{s}$ using the finer invarints of Betti numbers $(\a,\b)$, then each piece $\mathcal{M}_{\P^2}(\a,\b)^{s}$ is still unirational and we can thus produce a ``random" bundle that is presented by a matrix of forms of given degrees whenever possible. 

\bigskip

The results in this paper are implemented in the Macaulay2 \cite{M2} package \texttt{BundlesOnPn} \cite{MZ}, which generates all Betti numbers of bundles satisfying (\ref{Cohomology}) up to bounded regularity as well as ``random" bundles with given Betti numbers.

\subsection*{Acknowledgement}
The author thanks his advisor David Eisenbud for support and for pointing out that the earlier versions of the results may be generalized.

\section{Free resolutions of bundles}\label{FreeRes}
Throughout, we work over an algebraically closed field $k$. 
We fix $R := k[x_0,\dots,x_n]$ to denote the polynomial ring of $\P^n$.
For a coherent sheaf $\F$ on $\P^n$, we write $H^i_*(\F)$ for the $R$-module  $\bigoplus_{t\in\mathbb{Z}}H^i(\F(t))$.
We write $\VB$ for the set of isomorphism classes of bundles on $\P^n$ satisfying (\ref{Cohomology}).

\bigskip

We start with a standard observation on the relation between the vanishing of lower cohomologies of a coherent sheaf and the projective dimension of its section module.

\begin{proposition}\label{Resolution}
Let $M$ be a finitely generated graded $R$-module.
Then $\pdim_R M \le 1$ iff $M \cong H^0_*(\tilde{M})$ and $H^i_*(\tilde{M}) = 0$ for all $1\le i\le n-2$.
\end{proposition}
\begin{proof}
Let $H^i_m(-)$ denote the $i$-th local cohomology module supported at the homogeneous maximal ideal $m$ of $R$.
There is a four-term exact sequence
\[
0 \to H^0_m(M) \to M \to H^0_*(\tilde{M}) \to H^1_m(M) \to 0
\]
along with isomorphisms $H^{i+1}_m(M) \cong H^i_*(\tilde{M})$ for $1\le i\le n$.
By the vanishing criterion of local cohomology, we have $\depth M = \inf \{i \mid H^i_m(M) \ne 0\}$.
Finally, the Auslander-Buchsbaum formula states that $\pdim M = n+1-\depth M$.
The statement follows.
\end{proof}

\begin{definition}
Let $\E$ be a rank $r$ bundle on $\P^n$ satisfying (\ref{Cohomology}).
By \Cref{Resolution}, the $R$-module $H^0_*(\E)$ admits a unique (up to isomorphism) minimal graded free $R$-resolution 
\begin{align}
0 \to \bigoplus_{i = 1}^l R(-a_i) \xrightarrow{\phi} \bigoplus_{i = 1}^{l+r} R(-b_i) \to H^0_*(\E) \to 0. \tag{$*$}\label{ModuleRes}
\end{align}

We always arrange the numbers $a_1 \le \dots \le a_l$ and $b_1 \le \dots \le b_{l+r}$ in \textbf{ascending} order, and write $\a$ and $\b$ for brevity.
We call $(\a,\b)$ \emph{the Betti numbers} of $\E$.
\end{definition}

Note that $\E$ is isomorphic to a direct sum of line bundles iff $H^0_*(\E)$ is a free $R$-module iff $l = 0$ and the sequence $\a$ is empty.

\bigskip

The resolution (\ref{ModuleRes}) of graded $R$-modules sheafifies to a resolution
\begin{align}
0 \to \bigoplus_{i = 1}^l \O_{\P^n}(-a_i) \xrightarrow{\varphi} \bigoplus_{i = 1}^{l+r} \O_{\P^n}(-b_i) \to \E \to 0 \tag{$\star$}\label{SheafRes}
\end{align}
of $\E$ by direct sums of line bundles. 
Conversely, a resolution (\ref{SheafRes}) of $\E$ by direct sums of line bundles gives rise to a free resolution (\ref{ModuleRes}) of the $R$-module $H^0_*(\E)$ under the functor $H^0_*(-)$. 
With this understanding, we shall speak of these two resolutions of modules and sheaves interchangeably.
In particular, the morphism $\varphi$ is called \emph{minimal} iff the corresponding map of $R$-modules $\phi$ is minimal, i.e. $\phi\otimes_R k = 0$.

\bigskip

\subsection{Betti numbers}

In this subsection we classify the Betti numbers of bundles in $\VB$.

\bigskip

For a pair $(\a,\b)$ of finite sequences of integers in ascending order, we write $\VB(\a,\b)$ for the set of isomorphism classes of bundles with Betti numbers $(\a,\b)$.
For a sequence of integers $\d := (d_1, \dots, d_l)$, we define
\[
L(\d):= \bigoplus_{i = 1}^l R(-d_i) \text{ and } \L(\d) := \bigoplus_{i = 1}^l \O_{\P^n}(-d_i).
\]

\begin{theorem}\label{Betti}
Let $\a = (a_1, \dots, a_l)$ and $\b = (b_1, \dots, b_{l+r})$ be two sequences of integers in ascending order for some $l\ge 0$ and $r > 0$. 
The set $\VB(\a,\b)$ is nonempty iff $\a$ is empty or 
\begin{align*}
r \ge n \text{ and } a_i > b_{n+i} \text{ for } i = 1,\dots, l. \tag{A}\label{Admissible}
\end{align*}
In this case, the cokernel of $\varphi$ represents the class of a bundle in $\VB(\a,\b)$ for a general minimal map $\varphi \in \Hom(\L(\a),\L(\b))$.
\end{theorem}

This generalizes the results of Bohnhorst and Spindler \cite{BS} for $r = n$.
Likewise, we say a pair of ascending sequences of integers $(\a,\b)$ is \emph{admissible} if it satisfies the equivalent conditions of \Cref{Betti}.

\bigskip

The fact that a bundle $\E$ satisfying (\ref{Cohomology}) that is not a direct sum of line bundles must have rank $r\ge n$ also follows from the Evans-Griffith splitting criterion \cite[Theorem 2.4]{EG}.

\bigskip

In order to prove \Cref{Betti}, we need two lemmas regarding depth of minors of matrices.

\bigskip

Let $S$ denote a noetherian ring and let $\phi:S^p \to S^q$ be a map between two free $S$-modules.
For any integer $r$, the ideal $I_r(\phi)$ of $(r\times r)$-minors of $\phi$ is defined as the image of the map $\wedge^r S^p \otimes_S (\wedge^r S^q)^* \to S$, which is induced by the map $\wedge^r\phi:\wedge^r S^p \to \wedge^r S^q$.

Similarly, let $\varphi: \bigoplus_{i = 1}^p\O_{\P^n_A}(-a_i) \to \bigoplus_{i = 1}^q \O_{\P^n_A}(-b_i)$ be a morphism of sheaves on $\P^n_A$ over a noetherian ring $A$. Set $S:=A[x_0,\dots, x_n]$ and let $\phi: \bigoplus_{i = 1}^p S(-a_i) \to \bigoplus_{i = 1}^q S(-b_i)$ denote the corresponding morphism of graded free $S$-modules given by $H^0_*(\varphi)$. 
For any integer $r$, we define $I_r(\varphi) = I_r(\phi)$ as an ideal in $S$.

The depth of a proper ideal $I$ in a noetherian ring $S$ is defined to be the length of a maximal regular sequence in $I$. 
The depth of the unit ideal is by convention $+\infty$.
Recall that if $S$ is Cohen-Macaulay, then $\depth I = \codim I$ for every proper ideal $I$.

\begin{lemma}\label{DepthOpen}
Let $A$ be a finitely generated integral domain over $k$, and let $S$ be a finitely generated $A$-algebra. 
Suppose $\phi:S^q \to S^p$ is a morphism of free $S$-modules with $p\ge q$.
For a prime $P$ of $A$, let $\phi_P$ denote the morphism $\phi\otimes_A k(P)$ of free modules over the fiber ring $S\otimes_A k(P)$. 
For any integer $d$, the set of primes $P$ in $A$ such that $\depth I_q(\phi_P) \ge d$ is open in $A$.
\end{lemma}
\begin{proof}
Note that $I_q(\phi) = I_q(\phi^*)$. 
Let $\K_\bullet(\phi^*)$ be the Eagon-Northcott complex associated to $\phi^*$ as in \cite{EN}. 
Note that the formation of the Eagon-Northcott complex is compatible with taking fibers, i.e. $\K_{\bullet}(\phi^*)\otimes_A k(P) = \K_{\bullet}(\phi^*\otimes_A k(P))$.
For each prime ideal $P$ of $A$, we have $\depth I_q(\phi^*_P) \ge d$ iff $\K_\bullet(\phi^*)\otimes_A k(P)$ is exact after position $p-q+1-d$ by the main theorem in \cite{EN}.
The statement of the lemma follows from the general fact that the exactness locus of a family of complexes is open, see E.G.A IV 9.4.2 \cite{EGA}.
\end{proof}

\begin{lemma}\label{DepthZeros}
Let $S$ be a standard graded finitely generated $k$-algebra. 
Let $\phi:\bigoplus_{i =1}^q S(-a_i) \to \bigoplus_{i = 1}^p S(-b_i)$ be a morphism of graded free $S$-modules with $p\ge q$, and assume that $\phi$ is minimal, i.e. $\phi\otimes_S k = 0$. 
Suppose that relative to some bases, the matrix of $\phi$ has a block of zeros of size $u\times v$. 
Then $\codim I_q(\phi) \le p-q+1-\inf(u+v,p+1)+\inf(u+v,q)$.
\end{lemma}
\begin{proof}
For the case of generic matrices over a field, this is a result of Giusti-Merle \cite{GM} .
We fix, once for all, bases of the domain and target of $\phi$, and let $Z\subset \{1,\dots, p\} \times\{1,\dots, q\}$ be the $u\times v$ rectangle where the matrix of $\phi$ has zero entries.
Consider the polynomial ring  $A := k\left[\{x_{ij}\}^{1\le i\le p}_ {1\le j\le q}\right]/(x_{ij} \mid (i,j)\in Z)$, which is the coordinate ring of the affine space of $(p\times q)$-matrices with a zero block of size $u\times v$ in position $Z$.
Let $\psi:A(-1)^q\to A^p$ be the morphism given by the generic matrix $(x_{ij})$.
Then $\codim I_q(\psi) = p-q+1-\inf(u+v,p+1)+\inf(u+v,q)$ by \cite[Theorem 1.3]{GM}. 

The general case follows from Serre's result on the superheight of prime ideals in a regular local ring. 
The map $\phi$ corresponds to a morphism of $k$-algebras $A\to S$, where $x_{ij}$ is sent to the entry of the matrix of $\phi$ relative to the fixed bases.
In particular, note that $I_q(\phi) = SI_q(\psi)$.
Let $m$ and $m'$ denote the homogeneous maximal ideals of $A$ and $S$ respectively.
Since all entries of $\phi$ are in $m'$ by assumption, we have an induced morphism on the localizations $A_m \to S_{m'}$ where $S_{m'}m\subset m'$.
Let $P$ be a prime above $A_mI_q(\psi)$ of the least codimension.
Since $S_{m'}P\subset m'$,  Serre's result on superheight on prime ideals in a regular local ring \cite{Serre} implies that $\codim S_{m'}P \le \codim P$.
Now $I_q(\psi)$ and $SI_q(\psi)$ are homogeneous, and $S_{m'}I_q(\psi)\subset S_{m'}P$, therefore we conclude that
\begin{align*}
\codim SI_q(\psi) & =  \codim S_{m'}I_q(\psi) \\
& \le \codim S_{m'}P \\
& \le \codim P\\
&= \codim A_mI_q(\psi) \\
&= \codim I_q(\psi)\\
& = p-q+1-\inf(u+v,p+1)+\inf(u+v,q). \qedhere
\end{align*}
\end{proof}

\bigskip

The following is a simple fact that allows us to translate between bundles and homogeneous matrices whose ideals of maximal minors have maximal depth.

\begin{proposition}\label{Bundle}
Let $\a = (a_1,\dots, a_l)$ and $\b = (b_1,\dots, b_{l+r})$ for some $l>0$ and $r\ge 0$. 
For a map $\varphi \in \Hom(\L(\a),\L(\b))$, the cokernel of $\varphi$ is a rank $r$ bundle on $\P^n$ iff $\depth I_l(\varphi) \ge n+1$.
In this case, we have a resolution of $\E :=\coker \varphi$ by direct sums of line bundles
\[
0 \to \L(\a) \xrightarrow{\varphi} \L(\b) \to \E \to 0.
\]
\end{proposition}
\begin{proof}
The rank of $\coker \varphi$ is $r$ iff $I_r(\varphi)$ is nonzero iff $\varphi$ is injective at the generic point of $\P^n$ iff $\varphi$ is injective.
The ideal $I_r(\varphi)$ cuts out points on $\P^n$ where $\coker \varphi$ is not locally free of rank $r$.
Thus $\coker \varphi$ is a rank $r$ bundle iff $I_r(\varphi)$ is the unit ideal or is $m$-primary, where $m$ is the homogeneous maximal ideal of $R$. 
In either case $\depth I_r(\varphi) \ge n+1$.
\end{proof}

\begin{proof}[Proof of \Cref{Betti}]
If $\a$ is empty, then $\E  := \L(\b)$ has Betti numbers $(\a,\b)$. 
Suppose $\a$ is nonempty and $(\a,\b)$ satisfies condition (\ref{Admissible}). Consider the minimal map $\varphi:\L(\a) \to \L(\b)$ given by the following matrix
\[
\begin{blockarray}{ccccc}
& a_1 & \cdots & a_l & \\
\begin{block}{c(ccc)c}
b_1 & x_0^{a_1-b_1} & 0 & 0 & b_1\\
\vdots & \vdots &  \ddots & 0 &  \vdots \\
\vdots & \vdots &   & x_0^{a_l-b_l} & b_l \\
b_{n+1} & x_n^{a_1-b_{n+1}} &  & \vdots & \vdots \\
\vdots & 0 & \ddots & \vdots & \vdots \\
\vdots & 0 & 0 & x_n^{a_l-b_{l+n}} & b_{l+n}\\
\vdots & 0 & 0 & 0 & \vdots\\
b_{l+r} & 0 & 0 & 0  & b_{l+r}.\\
\end{block}
\end{blockarray}
\]
Since $\varphi$ drops rank nowhere on $\P^n$, we conclude that $\E := \coker \varphi$ is a rank $r$ bundle with a resolution by direct sums of line bundles
\[
0 \to \L(\a) \xrightarrow{\varphi} \L(\b) \to \E \to 0
\]
by \Cref{Bundle}.
Since $\varphi$ is minimal, it follows from \Cref{Resolution} that $\E \in \VB(\a,\b)$.

Conversely, suppose $\VB(\a,\b)$ is nonempty and $\a$ is nonempty.
Then there is a minimal map $\varphi \in \Hom(\L(\a),\L(\b))$  where $\coker \varphi$ is a rank $r$ bundle $\E$. 
Since $\varphi$ is minimal, it follows that $I_l(\varphi) \subset I_1(\varphi) \subset m$ is a proper ideal. 
By \Cref{Bundle}, we have $\depth I_l(\varphi) = n+1$. 
By the main theorem in \cite{EN}, we have $\depth I_l(\varphi) \le l+r-l+1 = r+1$.
It follows that we must have $r\ge n$.
Now suppose on the contrary that there is an index $1\le i\le l$ where $a_i \le b_{n+i}$. 
Since $\varphi$ is minimal, we see that the $(n+i,i)$-th entry in the matrix of $\varphi$ must be zero. 
In fact, since $\a$ and $\b$ are in ascending order, we must have a block of zeros
of size $(l+r-n-i+1)\times i$ as the following
\[
\begin{blockarray}{cccccc}
& a_1 & \cdots & a_i & \cdots & a_l  \\
\begin{block}{c(ccccc)}
b_1 &  & & &  &  \\
\vdots &  &   & & &   \\
\vdots &  &  & & &   \\
b_{n+i} & 0 & \cdots & 0 &  &  \\
\vdots & \vdots & & \vdots & & \\
\vdots & \vdots &  & \vdots & & \\
\vdots & \vdots &  & \vdots & & \\
b_{l+r} & 0 & \cdots & 0 & & \\
\end{block}
\end{blockarray}.
\]
By \Cref{DepthZeros}, we conclude that 
\begin{align*}
\depth I_l(\varphi) & \le l+r-l+1-\inf(l+r-n+1,l+r+1)+\inf(l+r-n+1,l)\\
& = r+1-(l+r-n+1)+l\\
& = n.
\end{align*}
This is a contradiction to the fact that $\depth I_l(\varphi) = n+1$.

Now we prove the last statement.
It is obvious when $\a$ is empty, so we assume $\a$ is nonempty. 
The set $\Hom(\L(\a),\L(\b))$ has the structure of the closed points of an affine space $\A^N$.
The subset of minimal maps is an affine subspace $\A^M$.
There is a tautological morphism $\Phi: \bigoplus_{i = 1}^l \O_{\P^n\times A^M}(-a_i) \to \bigoplus_{i = 1}^{l+r} \O_{\P^n\times \A^M}(-b_i)$, where the fiber $\Phi_P$ for a closed point $P$ of $\A^M$ is given by the minimal map that $P$ corresponds to.
By \Cref{DepthOpen}, the set $U$ of points in $\A^M$ where $\depth I_l(\Phi_P) \ge n+1$ is open. 
Since there is a morphism $\varphi \in \Hom(\L(\a),\L(\b))$ whose cokernel is a bundle $\E \in \VB(\a,\b)$, by \Cref{Bundle} the map $\varphi$ corresponds to a closed point in $U$.
It follows that $U$ is open and dense in $\A^M$.
\end{proof}

Recall that the category of bundles on $\P^n$ is a Krull-Schmidt category \cite{Atiyah}, i.e. every bundle $\E$ admits a decomposition $\E\cong \E_0\oplus \L$, unique up to isomorphism, where $\L$ is the direct sum of line bundles and $\E_0$ has no line bundle summands. 

\begin{corollary}\label{FreeRank}
Let $\E \in \VB(\a,\b)$ for some $\a$ nonempty.
If $\E \cong \E_0\oplus \L$ is the Krull-Schmidt decomposition of $\E$, then
$n\le \rank \E_0 \le \max\{j\mid a_l > b_{l+j}\}$.
\end{corollary}
\begin{proof}
Set $s := \max\{j\mid a_l > b_{l+j}\}$ and define $\b' := b_1,\dots, b_s$.
Let $\pi:\L(\b) \to \L(\b')$ be the coordinate projection. 
If $\varphi \in \Hom(\L(\a),\L(\b'))$ is a minimal map whose cokernel is a bundle $\E$, then we claim that $\varphi' := \pi\circ\varphi$ is a minimal map in $\Hom(\L(\a),\L(\b))$ whose cokernel is a bundle $\E'$.
To see this, observe that since $a_l\le b_{l+i}$ for $s<i\le r$ and $\varphi$ is minimal, the last $(r-s)$ rows of the matrix representing $\varphi$ relative to any bases are zero.
In particular, we have $I_l(\varphi) = I_l(\pi\circ \varphi)$.
By \Cref{Bundle}, the cokernel of $\varphi'$ is a bundle.
It follows from the snake lemma that $\E \cong \E'\oplus \L$, where $\L$ is the kernel of the projection $\pi$.
This shows that $\rank \E_0 \le s$.
Observe that $\E_0$ also satisfies (\ref{Cohomology}) and thus $\rank \E_0 \ge n$ by \Cref{Betti}.
\end{proof}

\bigskip

\subsection{Finiteness} In this subsection, we show that there are only finitely many possible Betti numbers of bundles in $\VB$ with given rank, first Chern class and bounded regularity.

\bigskip

Recall that a coherent sheaf $\F$ on $\P^n$ is said to be $d$-regular if $H^i(\F(d-i)) = 0$ for all $i>0$.
The Castelnuovo-Mumford regularity of $\F$ is the least integer $d$ that $\F$ is $d$-regular.
By the semicontinuity of cohomologies, being $d$-regular is an open condition for a family of coherent sheaves on $\P^n$.
The notion of regularity also exists for graded $R$-modules.
See \cite{DE} for an exposition.

\bigskip

If $\E \in \VB(\a,\b)$, then $\reg \E = \max(b_{l+r}, a_l-1)$.
Since the regularity depends only on the Betti numbers, we define $\reg (\a,\b) := \max(b_{l+r},a_l-1)$ for any admissible pair $(\a,\b)$.

\begin{proposition}\label{Finite}
There are only finitely many possible Betti numbers $(\a,\b)$ of rank $r$ bundles on $\P^n$ satisfying (\ref{Cohomology}) with fixed first Chern class $c_1$ and regularity $\le d$.
\end{proposition}
\begin{proof}
Since $c_1 = \sum_{i = 1}^l a_i - \sum_{i = 1}^{l+r} b_i$, the statement is evidently true for direct sums of line bundles. 
Thus we may consider the case $l>0$.
Since $a_i$ and $b_i$ are bounded above by $d+1$, we only need to show that $l$ is bounded above and $b_1$ is bounded below.
Consider the following inequalities
\begin{align*}
l &\le \sum_{i = 1}^l (a_i-b_{i+n}) \\
& = c_1+\sum_{i = 1}^n b_i +\sum_{i = l+n+1}^{l+r} b_i\\
& \le c_1 + r \cdot d.
\end{align*}
And similarly,
\begin{align*}
b_1 &= -c_1-\sum_{i = 2}^n b_i - \sum_{i = l+n+1}^{l+r} b_i +\sum_{i =1}^l (a_i-b_i)\\
& \ge -c_1 - (r-1)\cdot d + l. \qedhere
\end{align*}
\end{proof}

This generalizes the observation of Dionisi-Maggesi \cite{DM} for the case $n = r = 2$.

\bigskip

\subsection{Hilbert functions of bundles}\label{BundleSequence}

In this subsection, we classify the Hilbert functions of bundles in $\VB$.
We introduce an efficient way to represent and generate them.

\bigskip

Recall that the \emph{Hilbert function} of a bundle $\E$ on $\P^n$ is the function $H_\E(t):\mathbb{Z} \to \mathbb{Z}$ given by $H_\E(t) = \dim_k H^0(\E(t))$. For any function $H:\mathbb{Z}\to \mathbb{Z}$, we define $\VB(H)$ to be the subset of $\VB$ consisting of isomorphism classes of bundles with Hilbert function $H$.

\begin{definition}
The \emph{numerical difference} of a function $H:\mathbb{Z}\to \mathbb{Z}$ is a function $\partial H:\mathbb{Z} \to \mathbb{Z}$ given by $\partial H(t) := H(t)-H(t-1)$.
We inductively define $\partial^{i+1}H := \partial \partial^i H$. 
\end{definition}

Note that if $H:\mathbb{Z}\to \mathbb{Z}$ is a function such that $H(t) = 0$ for $t\ll 0$, then $H$ can be recovered by its $i$-th difference $\partial^i H$ for any $i\ge 0$.

\begin{theorem}\label{Hilbert}
A function $H:\mathbb{Z}\to \mathbb{Z}$ is the Hilbert function of a rank $r$ bundle $\E\in \VB$ if and only if
\begin{enumerate}
\item $\partial^nH(t) = 0$ for $t\ll 0$ and $\partial^nH(t) = r$ for $t\gg 0$,
\item $\partial^nH(t+1) < \partial^nH(t)$ implies that $\partial^nH(t+1)\ge n$.
\end{enumerate}
\end{theorem}
\begin{proof} Let $\mu(\d,t)$ denote the number of times an integer $t$ occurs in the sequence $\d$.

($\Longrightarrow$): Suppose $\E$ is a rank $r$ bundle in $\VB(H)$.
The Grothendieck-Riemann-Roch formula states that
\[
\chi(\E(t)) = \int_{\P^n} \op{ch}(\E(t))\cdot \op{td}(T_{\P^n})\label{RiemannRoch}. 
\]
A routine computation shows that the leading coefficient of the Hilbert polynomial $\chi(\E(t))$ is $r\cdot t^n /n!$. 
Since the Hilbert function $H$ eventually agrees with the Hilbert polynomial, we see that $\partial^nH(t) = 0$ for $t\ll 0$ and $\partial^n H(t) = r$ for $t\gg 0$. 

Let $(\a,\b)$ be the Betti numbers of $\E$. 
If $\a$ is empty, then $\E$ is a direct sum of line bundles and $\partial^nH$ is monotone nondecreasing and thus satisfies both conditions.
We prove the case where $\a$ is non-empty.
Consider the minimal free resolution
\[
0 \to L(\a) \to L(\b) \to H^0_*(\E) \to 0.
\]
A simple calculation shows that $\partial^{n+1} H(R(-a),t)$ is the delta function at $a$. 
It follows from the minimal resolution that $\partial^{n+1} H(t) = \mu(\b,t)-\mu(\a,t)$.
Suppose $\partial^nH(t+1) < \partial^nH(t)$ for some $t$, then $\partial^{n+1}H(t+1) <0$ and thus $\mu(\a,t+1)>0$. 
Let $j$ be the largest index where $a_j = t+1$.
By \Cref{Betti}, we have $a_j>b_{j+n}$ and therefore
\[
\partial^nH(t+1) = \sum_{i\le t+1} \partial^{n+1}H(i) = \sum_{i\le t+1}(\mu(\b,i)-\mu(\a,i)) \ge j+n-j = n.
\]

($\Longleftarrow$): Conversely, suppose $H$ satisfies the conditions of the theorem.
We define the ascending sequences of integers $\aa$ and $\bb$ by the property that for all $t\in \mathbb{Z}$,
\[
\mu(\aa,t) = \max\{0, \partial^nH(t-1)-\partial^nH(t)\},\quad \mu(\bb,t) = \max\{0, \partial^n H(t-1) -\partial^n H(t)\}.
\]
By the first condition on $H$, the sequences $\aa$ and $\bb$ are finite.
Furthermore, if $\aa$ has length $l$ then $\bb$ has length $l+r$.
The second condition on $H$ implies that $a_i \ge b_{i+n}$ for all $1\le i\le l$.
Since $\aa$ and $\bb$ share no common entries by construction, it follows that $a_i > b_{i+n}$ for all $1\le i\le l$.
By \Cref{Betti}, there is a rank $r$ bundle $\E$ on $\P^n$ satisfying (\ref{Cohomology}) with Betti numbers $(\aa,\bb)$.
The Hilbert function of $\E$ is $H$ by the reasoning of the previous direction.
\end{proof}

The above theorem suggests that we use the finitely many intermediate values of $\partial^n H$ to encode the infinitely many values of the Hilbert function $H$.

\begin{definition}
A finite sequence of integers $\B = B_1,\dots, B_m$ for some $m\ge 1$ is called a \emph{bundle sequence of rank $r$} if it satisfies the following
\begin{enumerate}
\item $B_i >0$ for $1\le i\le m$,
\item $B_m = r$ and $B_{m-1} \ne r$, 
\item $B_{i+1}<B_i$ implies $B_{i+1} \ge n$.
\end{enumerate}

If $\E$ is a rank $r$ bundle in $\VB(H)$ for some Hilbert function $H$, then we set 
\[
s_0 := \inf\{t \mid \partial^n H(t)\ne 0\},\quad s_1 := \sup \{t \mid \partial^n H(t) \ne r\}.
\]
The sequence $\partial^nH(s_0), \partial^nH(s_0+1),\dots, \partial^nH(s_1+1)$ is a bundle sequence of rank $r$ by \Cref{Hilbert}, which we call the \emph{bundle sequence of $H$ and of $\E$}. 
\end{definition}

By \Cref{Hilbert}, there is a one-to-one correspondence between the set of Hilbert functions of rank $r$ bundles in $\VB$ \textbf{up to shift} and the set of bundles sequences of rank $r$.
The ambiguity of shift disappears if we deal with normalized bundles.

\begin{definition}
We say a rank $r$ bundle on $\P^n$ is \emph{normalized} if $-r < c_1(\E) \le 0$.
Since $c_1(\E(t)) = c_1(\E)+r\cdot t$, it follows that every bundle can be normalized after twisting by the line bundle $\O(-\lceil c_1(\E)/r\rceil)$.
\end{definition}

We define the degree of a bundle sequence $\B = B_1,\dots, B_m$, denoted by $\deg \B$, to be the sum $B_1+\dots+B_m$.

\begin{proposition}\label{Regularity}
If a normalized rank $r$ bundle $\E\in \VB$ has bundle sequence $\B$, then $\reg \E \ge \lceil\deg \B / r \rceil-2$.
\end{proposition}
\begin{proof}
Suppose $\E$ has Betti numbers $(\a,\b)$ and Hilbert function $H$.
We set $c := \max(a_l,b_{l+r})$ and $s_1 := \sup\{t\mid \partial^n H(t) \ne r\}$.
It follows from the short exact sequence
\[
0 \to L(\a) \to L(\b) \to H^0_*(\E) \to 0
\]
that $s_1<c$.
We have
\begin{align*}
c_1(\E) & = \sum_{i = 1}^l a_i-\sum_{i = 1}^{l+r} b_i = -\sum_t t \cdot \partial^{n+1}H(t) = -\sum_t t \cdot (\partial^n H(t)-\partial^nH(t-1))\\
& = \sum_{t\le s_1+1} t \cdot \partial^n H(t-1) - \sum_{t\le s_1+1} t \cdot \partial^n H(t)\\
& = \sum_{t \le s_1} \partial^n H(t) - (s_1+1)\cdot r = \deg \B - (s_1+2)\cdot r \ge \deg \B -(c+1)\cdot r.
\end{align*}
Since $\E$ is normalized, we must have $c \ge \lceil \deg \B / r\rceil-1$.
Finally, regularity $\E$ is $c$ or $c-1$ depending on whether $b_{l+r} \ge a_l-1$ or not.
\end{proof}

\begin{proposition}\label{Inductive}
If $\B = B_1,\dots, B_m$ is a bundle sequence of rank $r$ and degree $d$, then $\B' = B_2,\dots, B_m$ is a bundle sequence of rank $r$ and degree $d-B_1$.
\end{proposition}

It follows from \Cref{Regularity} and \Cref{Inductive} that we can inductively generate, in the form of bundle sequences, all Hilbert functions of normalized bundles satisfying (\ref{Cohomology}) up to any bounded regularity.
The generation is reduced to a partition problem with constraints.

\begin{example}\label{Ex1}
The following are all bundles sequences of rank $4$ and degree $9$ on $\P^3$
\[
\{(1^5,4),(1^3,2,4),(1^2,3,4),(1,2^2,4),(2,3,4),(5,4) \}.
\]
Here we use $t^j$ to denote the sequence of $j$ copies of $t$.
\end{example}

\bigskip

\subsection{Semistability} \label{Semistability}

In this subsection, we address the following question. Do the Betti numbers determine the semistability of a bundle in $\VB$? If so, what is the criterion?

\bigskip

Here we use $\mu$-semistability, where $\mu(\F) := c_1(\F)/\op{rank}(\F)$ for any torsion-free coherent sheaf $\F$ on $\P^n$.
The results are similar for Hilbert polynomial semistability as in \cite{Maruyama}.

\bigskip

For $r<n$, all rank $r$ bundles $\E$ satisfying (\ref{Cohomology}) are direct sums of line bundles by \Cref{Betti}, which are not semistable except for $\O(d)^r$.
The main result in \cite{BS} states that if $\E$ satisfies (\ref{Cohomology}) and has rank $r = n$, then $\E$ is semistable iff its Betti numbers $(\a,\b)$ satisfy $b_1\ge \mu(\E) = (\sum_{i = 1}^l a_i-\sum_{i = 1}^{l+n} b_i)/n$.
The latter condition is obviously necessary.

\bigskip

The following example demonstrates that for $r>n$, the semistability of a bundle in $\VB$ is not determined by its Betti numbers in general.

\begin{example}
For any $r>n$, consider $(\a,\b)$ where
\[
a_1 = 2,\quad b_i = \left\{\begin{array}{l l} 0 & 1\le i < r\\
1 & r\le i \le r+1.\end{array}\right.
\]
Let $\varphi$ and $\psi$ be two maps in $\Hom(\L(\a),\L(\b))$ defined by the matrices
\[
(0,\dots, x_0^2,\dots,x_{n-1}^2, x_n, 0)^T,\quad (0,\dots, x_0^2,\dots, x_{n-2}^2, x_{n-1}, x_n)^T
\]
respectively.
Then $\E_1 := \coker \varphi$ and $\E_2 := \coker \psi$ are rank $r$ bundles satisfying (\ref{Cohomology}) with Betti numbers $(\a,\b)$ by \Cref{Bundle}.
Furthermore, it is easy to see that $\E_1 \cong \E_1'\oplus \O(-1) \oplus \O^{r-n-1}$ and $\E_2 \cong \E_2'\oplus \O^{r-n}$ for some rank $n$ bundles $\E_1'$ and $\E_2'$ respectively.
Since $\mu(\E_1) = \mu(\E_2) = 0$, it is clear that $\E_1$ is not semistable.
On the other hand, the bundle $\E_2'$ is semistable by the criterion for the case $r = n$ stated above.
Since both $\E_2'$ and $\O^{r-n}$ are semistable bundles with $\mu = 0$, it follows that so is $\E_2$. 
\end{example}

\bigskip

The main reason to discuss semistability is that we might hope for a coarse moduli structure on the set $\VB(\a,\b)$. 
However, the above example illustrates the difficulty. 
In \Cref{Moduli} we will define a topology on $\VB(\a,\b)$, where the semistable bundles form an open subspace $\VB(\a,\b)^{ss}$. 
The space $\VB(\a,\b)^{ss}$ supports the structure of a subscheme of $\mathcal{M}(\chi)$, the coarse moduli space of semistable torsion-free sheaves with Hilbert polynomial $\chi$, whose existence is established by Maruyama \cite{Maruyama}.


\bigskip

\section{The Betti number stratification}\label{Stratification}

The set $\VB$ is the disjoint union of $\VB(H)$ for all possible Hilbert functions $H$ which are classified by \Cref{Hilbert}. 
In this section we define a natural topology on $\VB(H)$ and study how $\VB(H)$ is stratified by bundles with different Betti numbers.
In the following, we fix a Hilbert function $H$ satisfying the conditions of \Cref{Hilbert}.

\bigskip

\subsection{The graded lattice of Betti numbers}\label{Lattice}

In this subsection we show that all possible Betti numbers of bundles in $\VB(H)$ form a graded lattice, such that those with bounded regularity form a finite sublattice.

\begin{definition}
We define $\Betti(H)$ to be the set of Betti numbers $(\a,\b)$ of bundles in $\VB(H)$.
There is a grading $\Betti(H) = \bigsqcup_q \Betti^q(H)$, where
\[
\Betti^q(H) := \{(\a,\b)\in \Betti(H)\mid \text{$\a$ and $\b$ have exactly $q$ entries in common}\}.
\]
\end{definition}

We remark that $\Betti(H)$ is infinite in general without restrictions on regularity. 
This is due to the fact that the Hilbert function $H$ only bounds regularity from below (see \Cref{Regularity}) but not above, as the following example demonstrates.
 
\begin{example}
Let $(\a,\b)\in \Betti(H)$. 
For some arbitrarily large integer $c$, regarded as a singleton sequence, the pair $(\a,\b)+c$ is admissible by \Cref{Betti}.
Note that any bundle with these Betti numbers has a line bundle summand by \Cref{FreeRank}.
\end{example}

\begin{proposition}
There is a unique element in $\Betti^0(H)$, which we denote by $(\aa,\bb)$.
\end{proposition}
\begin{proof}
The construction of an element in $\Betti^0(H)$ is given in the proof of \Cref{Hilbert}.
Recall from the proof of \Cref{Hilbert} that $\partial^{n+1}H(t)= \mu(\aa,t)-\mu(\bb,t)$. 
The uniqueness of $(\aa,\bb)$ follows from the fact that either $\mu(\aa,t) = 0$ or $\mu(\bb,t) = 0$ by assumption.
\end{proof}

\bigskip

We now define a partial order on all pairs of increasing sequences of integers.

\begin{definition}
Let $\a,\b,\c$ be three finite sequences of integers in ascending order. 
The sum $\a+\c$ is defined be the sequence obtained by appending $\c$ to $\a$ and sorting in ascending order.
It is clear that this operation is associative.

We define $(\a,\b)+\c$ to be the pair $(\a+\c, \b+\c)$.
If $(\a',\b') = (\a,\b)+\c$ for some $\c$, then we say $(\a,\b)$ is a \emph{generalization} of $(\a',\b')$ and write $(\a,\b)\preceq (\a',\b')$. 
\end{definition}

\bigskip

A direct consequence of \Cref{Betti} is that admissibility is stable under generalization.

\begin{lemma}\label{Generalization}
If $(\a,\b)\preceq (\a',\b')$ and $(\a',\b')$ is admissible, then so is $(\a,\b)$. 
\end{lemma}
\begin{proof}
By induction, it suffices to prove the case where $\a'$ and $\b'$ have a common entry $c$ at index $p$ and $q$ respectively, and that $(\a,\b)$ is obtained from $(\a',\b')$ by removing $a'_p$ and $b'_q$.
We may assume that $p$ and $q$ are the largest indices where $a'_p = c$ and $b'_q = c$ respectively. 
For $i < p$, we have $a_i = a'_i$.
But $i+n <q$ and $b_{i+n} = b'_{i+n}$ for $i<p$ since $q > n+p$ by \Cref{Betti}.
Therefore $a_i > b_{i+n}$ for $i<p$.
In this case, $b_{l+n} = b_{l+n}'$ and $a_l > b_{l+n}$. 
For $i>p$, we have $a_{i-1} = a'_i > c$.
In this case, either $i+n \le q$, in which case $b_{i+n-1} \le c < a_{i-1}$;
or $i+n>q$, and $b_{i+n-1} = b'_{i+n}$ thus $b_{i+n-1} < a_{i-1}$. 
We conclude that $(\a,\b)$ is also admissible.
\end{proof}

\begin{corollary}\label{Minimum}
Every $(\a,\b)$ in $\Betti(H)$ is of the form $(\aa,\bb)+\c$ for some $\c$.
\end{corollary}

\bigskip

The main theorem of this subsection is the following.

\begin{theorem}\label{GradedLattice}
The set $\Betti(H)$ has the structure of a graded lattice given by the partial order $\preceq$ and the grading $\Betti(H) = \bigsqcup_q \Betti^q(H)$.
\end{theorem}

For the clarity of the proof, we establish the existence of meet in several steps.

\begin{lemma}\label{Diamond}
If $c$ and $d$ are two distinct integers (considered as singleton sequences) such that both $(\a,\b)+c$ and $(\a,\b)+d$ are admissible, then so is $(\a,\b)+c+d$.
\end{lemma}
\begin{proof}
The lemma is simple, but the notations may make it appear more complicated than it is.
Notheless, we include a proof here for the sake of completeness.

For an ascending sequence $\d$ and an integer $t$, let $p(\d,t)$ denote the largest index $i$ where $d_i = t$. 
We may assume $c<d$, and write $(\a',\b') := (\a,\b)+c$, $(\a'',\b'') := (\a,\b)+d$ and $(\a''',\b''') := (\a,\b)+c+d$.

Since $(\a',\b')$ is admissible, we have $p(\a',c)<p(\b',c)-n$.
Since $c<d$, it follows that $p(\a''',c) = p(\a',c)$ and $p(\b''',c) = p(\b',c)$. We conclude that $p(\a''',c) < p(\b''',c)-n$.
Since $(\a'',\b'')$ is admissible, we have $p(\a'',d)<p(\b'',d)-n$.
Since $c<d$, it follows that $p(\a''',d) = p(\a'',d)+1$ and $p(\b''',d) = p(\b'',d)+1$.
We conclude that $p(\a''',d) < p(\b''',d)-n$.
Finally, we show that $(\a''',\b''')$ is admissible.
For $i<p(\a''',d)$, we have $i+n<p(\b''',d)$ and thus $a'''_i = a_i' > b_{i+n}' = b'''_{i+n}$.
For $p(\b''',c)-n < i$, we have $p(\a''',c)<i$ and thus $a'''_i = a_{i-1}'' > b_{i+n-1}'' = b'''_{i+n}$.
For $p(\a''',d)\le i \le p(\b''',c)-n$, we have $a_i \ge d > c \ge b_{i+n}$.
\end{proof}

\begin{lemma}\label{Ladder}
If $\c$ is an integer sequence and $d$ is an integer (considered as a singleton sequence) not appearing in $\c$, such that both $(\a,\b)+\c$ and $(\a,\b)+d$ are admissible, then so is $(\a,\b)+\c+d$.
\end{lemma}
\begin{proof}
By \Cref{Generalization} and \Cref{Diamond}, the pair $(\a,\b)+c_1+d$ is admissible.
Applying \Cref{Diamond} again with $(\a,\b)+c_1$ in place of $(\a,\b)$, we see that $(\a,\b)+c_1+c_2+d$ is admissible.
By induction it follows that $(\a,\b)+\c+d$ is admissible.
\end{proof}

\begin{proof}[Proof of \Cref{GradedLattice}]
If $(\a',\b')\in \Betti^i(H)$ and $(\a,\b)\in \Betti^j(H)$ such that $(\a',\b') =(\a,\b)+\c$ for some $\c$, then obviously $i\ge j$.
The cover relations in $\Betti(H)$ are given exactly by adding singleton sequences. 
It follows that $(\Betti(H),\preceq)$ is a graded poset.

Suppose $(\a,\b)$ and $(\a',\b')$ are in $\Betti(H)$.
By \Cref{Minimum}, there are sequences $\c$ and $\c'$ such that $(\a,\b) = (\aa,\bb)+\c$ and $(\a',\b') = (\aa,\bb)+\c'$.
We define $\min(\c,\c')$ to be the descending integer sequence where an integer $t$ occurs $\min(\mu(\c,t),\mu(\c',t))$ times, and similarly for $\max(\c,\c')$.

Clearly $(\a,\b)+\min(\c,\c') \preceq (\a,\b)+\c$ and thus is admissible by \Cref{Generalization}.
It follows that  $(\a,\b)+\min(\c,\c')$ is the meet of $(\a,\b)$ and $(\a',\b')$ in $\Betti(H)$.

We claim that $(\a,\b)+\max(\c,\c')$ is admissible, and thus it is the join of $(\a,\b)$ and $(\a',\b')$ in $\Betti(H)$.
To see this, we may replace $(\a,\b)$ by $(\a,\b)+\min(\c,\c')$ and assume that $\c$ and $\c'$ have no common entries.
By \Cref{Ladder}, we see that $(\a,\b)+\c+c_1'$ is admissible.
Applying \Cref{Ladder} again with $(\a,\b)+c_1'$ in place of $(\a,\b)$, we conclude that $(\a,\b)+c_1'+\c+c_2'$ is admissible.
By induction, it follows that $(\a,\b)+\c'+\c$ is admissible.
\end{proof}

For any integer $d$, let $\Betti(H)_{\le d}$ denote the subset of Betti numbers of bundles that are $d$-regular. 
The set $\Betti(H)_{\le d}$ inherits a grading $\bigoplus_{q\ge 0} \Betti^q(H)_{\le d}$, where $\Betti^q(H)_{\le d} := \Betti^q(H)\cap \Betti(H)_{\le d}$.

\begin{corollary}\label{Sublattice}
For any integer $d$, the set $\Betti(H)_{\le d}$ is a finite graded lattice isomorphic to the lattice of subsequences of some sequence $\c$.
\end{corollary}
\begin{proof}
If $(\a,\b)\preceq (\a',\b')$, then $\reg (\a,\b)\le (\a',\b')$. 
If $(\a'',\b'')$ is the join of $(\a,\b)$ and $(\a',\b')$ in $\Betti(H)$, then the regularity of $(\a'',\b'')$ is the maximum of those of $(\a,\b)$ and $(\a',\b')$ by the construction in the proof of \Cref{GradedLattice}.
It follows that $\Betti(H)_{\le d}$ is a graded lattice.
The finiteness of $\Betti(H)_{\le d}$ follows from \Cref{Finite}.
Thus there is a maximum element of the form $(\aa,\bb)+\c$ for some sequence $\c$. By \Cref{Generalization}, we see that
\[
\Betti^q(H)_{\le d} = \{(\aa,\bb)+\c'\mid \c' \text{ is a subsequence of $\c$ of length $q$}\}. \qedhere
\]
\end{proof}

\begin{example}
Let $H$ be the Hilbert function of a normalized bundle on $\P^3$ with bundle sequence $(5,4)$.
With the same notation as in \Cref{Ex1}, the minimal element of $\Betti(H)$ is given by $\aa = (0)$ and $\bb = (-1^5)$.
The maximum element of $\Betti(H)_{\le 2}$ is $(\aa,\bb)+\c$, where $\c = (0,1,2)$.
In particular, 
\[
\Betti^q(H)_{\le 2} = \{(\aa,\bb)+\c'\mid \c' \text{ is a subsequence of $(0,1,2)$ of length $q$}\}
\]
and $\Betti(H)_{\le 2}$ is isomorphic to the lattice of subsequences of $(0,1,2)$.
\end{example}


\bigskip

\subsection{The stratification}\label{Moduli}

In this subsection, we define a natural topology on $\VB(H)$.
We then describe the stratification of $\VB(H)$ by locally closed subspaces $\VB(\a,\b)$.

\bigskip

\begin{definition}
Let $(\a,\b)\in \Betti(H)$. 
Let $\A(\a,\b)$ denote the structure of the affine space on the vector space $\Hom(\L(\a),\L(\b))$.
The minimal maps form an affine subspace $\A^0(\a,\b)$ in $\A(\a,\b)$.
We define the subset of matrices whose maximal minors have maximal depth
\[
\M(\a,\b) := \{\varphi\in \A(\a,\b)\mid \depth I_l(\varphi) \ge n+1\},
\]
\[
\M^0(\a,\b) := \{\varphi\in \A^0(\a,\b) \mid \depth I_l(\varphi) \ge n+1\}.
\] 
As in the proof of \Cref{Betti}, the subset $\M(\a,\b)$ and $\M^0(\a,\b)$ are open subvarieties of $\A(\a,\b)$ and $\A^0(\a,\b)$ respectively.
For $\A = \A(\a,\b)$ and $\A^0(\a,\b)$, the tautological morphism 
\[
\Phi: \bigoplus_{i = 1}^l \O_{\P^n_\A}(-a_i) \to \bigoplus_{i = 1}^{l+r} \O_{\P^n_\A}(-b_i)
\]
gives a tautological family of sheaves $\E := \coker \Phi$ over $\A$, which pulls back to a family of bundles $\E(\a,\b)$ and $\E^0(\a,\b)$ satisfying (\ref{Cohomology}) over $\M(\a,\b)$ and $\M^0(\a,\b)$ respectively by \Cref{Bundle}.
\end{definition}

Let $G(\a,\b)$ denote the algebraic group $\Aut(\L(\a))\times \Aut(\L(\b))$.
The natural action $\rho: G(\a,\b)\times \Hom(\L(\a),\L(\b)) \to \Hom(\L(\a),\L(\b))$ given by $(f,g) \times \varphi \mapsto  f \circ \varphi \circ g$
is a morphism of algebraic varieties.
The action $\rho$ leaves the subspace of minimal maps invariant.
Since the change of coordinates does not change the ideal of maximal minors, it follows that the open subvarieties $\M(\a,\b)$ and $\M^0(\a,\b)$ are stable under the $G(\a,\b)$-action.

\begin{lemma}\label{Orbit}
Two maps $\varphi, \psi \in \M(\a,\b)$ are in the same $G(\a,\b)$-orbit iff $\coker \varphi \cong \coker \psi$.
\end{lemma}
\begin{proof}
Clearly if $\varphi, \psi$ are in the same $G(\a,\b)$-orbit then $\coker \varphi \cong \coker \psi$.
Conversely, let $\E := \coker \varphi$ and $\E' := \coker \psi$.
Then the isomorphism of the $R$-modules $H^0_*(\E) \cong H^0_*(\E')$ lifts to an isomorphism of free resolutions
\[
\begin{tikzcd}
0 \arrow[r] & L(\a)\arrow[d,"f" ',"\cong"] \arrow[r,"\varphi"] & L(\b) \arrow[d,"g" ', "\cong"] \arrow[r] & H^0_*(\E) \arrow[d,"\cong"] \arrow[r] & 0\\
0 \arrow[r] & L(\a) \arrow[r,"\varphi'"] & L(\b) \arrow[r] & H^0_*(\E') \arrow[r] & 0.
\end{tikzcd}
\]
It follows that $\varphi, \varphi'$ are in the same $G(\a,\b)$-orbit. 
\end{proof}

\Cref{Bundle} and \Cref{Orbit} imply that the set $\VB(\a,\b)$ supports the structure of the quotient topological space $\M^0(\a,\b) / G(\a,\b)$.
Similarly, we let $\VB(\a,\b)^{\preceq}$ denote the subset of $\VB$ consisting of isomorphism classes of bundles $\E$ that admit a (not necessarily minimal) free resolution of the form
\[
0 \to L(\a) \to L(\b) \to H^0_*(\E) \to 0.
\]
Then \Cref{Orbit} also implies that the set $\VB(\a,\b)^{\preceq}$ supports the structure of the quotient topological space $\M(\a,\b) / G(\a,\b)$. 
Clearly the inclusion of sets $\VB(\a,\b) \subseteq \VB(\a,\b)^{\preceq}$ is an inclusion of topological spaces. 

\bigskip

\begin{lemma}\label{Subspace}
If $(\a,\b) \preceq (\a',\b')$ in $\Betti(H)$, then $\VB(\a,\b)^{\preceq}$ is a subspace of $\VB(\a',\b')^{\preceq}$.
In particular, $\VB(\a,\b)$ is a subspace of $\VB(\a',\b')^{\preceq}$.
\end{lemma}
\begin{proof}
Let $(\a',\b') = (\a,\b)+\c$ for some $\c$. 
Consider an injective morphism $\iota: \M(\a,\b) \to \M(\a',\b')$ given by $\varphi \mapsto \varphi \oplus \op{Id}_{\L(\c)}$.
It is not hard to see that the ideal of maximal minors does not change under this map, and thus $\iota$ is well-defined.
Suppose $\varphi, \psi$ are two morphisms in $\M(\a,\b)$ such that $\varphi \oplus \op{Id}_{\L(\c)}$ and  $\psi \oplus \op{Id}_{\L(\c)}$ are in the same $G(\a',\b')$-orbit.
It follows that $\coker \varphi \oplus \op{Id}_{\L(\c)} \cong \coker  \psi \oplus \op{Id}_{\L(\c)}$. 
Since $\coker \varphi \cong \coker \varphi \oplus \op{Id}_{\L(\c)}$ and $\coker \psi \cong \coker \psi \oplus \op{Id}_{\L(\c)}$, we conclude that $\coker \varphi \cong \coker \psi$. 
It follows from \Cref{Orbit} that $\varphi$ and $\psi$ are in the same $G(\a,\b)$-orbit.
This shows that the composition 
\[
\M(\a,\b)\to \M(\a',\b') \to \VB(\a',\b')^{\preceq}
\]
 induces an injection of topological spaces on the quotient 
$\VB(\a,\b)^{\preceq} \hookrightarrow \VB(\a',\b')^{\preceq}$.
\end{proof}

\bigskip

For each integer $d$, the set $\Betti(H)_{\le d}$ is a lattice by \Cref{Sublattice} and thus has a maximum element $(\a',\b')$.
It follows from \Cref{Subspace} that every $d$-regular bundle $\E$ in $\VB(H)$ admits a (not necessarily minimal) free resolution of the form
\[
0 \to L(\a') \to L(\b') \to H^0_*(\E) \to 0.
\] 
Let $\VB(H)_{\le d}$ be the subspace of $\VB(H)$ consisting of isomorphism classes of $d$-regular bundles. 
Then by \Cref{Orbit}, the set $\VB(H)_{\le d}$ supports the structure of the quotient topological space $\M(\a',\b')/G(\a',\b')$.

It follows from \Cref{Subspace} and the construction above that if $d<d'$, then $\VB(H)_{\le d}$ is a subspace of $\VB(H)_{\le d'}$.
Finally, we define a topology on $\VB(H)$ by
\[
\VB(H) = \varinjlim_d \VB(H)_{\le d}.
\]

\begin{proposition}\label{Dense}
For each integer $d$, the subspace $\VB(H)_{\le d}$ is open in $\VB(H)$.
\end{proposition}
\begin{proof}
We need to show that $\VB(H)_{\le d}$ is open in $\VB(H)_{\le d'}$ for $d'\gg 0$.
Let $(\a',\b')$ be the maximum element in $\Betti(H)_{\le d'}$, and consider the quotient map $\pi:\M(\a',\b') \to \VB(H)_{\le d'}$.
By the semicontinuity of cohomologies, the fibers of the tautological family $\E(\a',\b')$ are $d$-regular over an open subset of $\M(\a',\b')$. 
It follows that $\VB(H)_{\le d}$ is the image of this open subset under $\pi$, and thus is an open subspace of $\VB(H)_{\le d'}$.
\end{proof}

\begin{proposition}
The topological space $\VB(H)$ is irreducible and unirational.
\end{proposition}
\begin{proof}
For $d\gg 0$, the subspace $\VB(H)_{\le d}$ is dense in $\VB(H)$. 
Since $\VB(H)_{\le d}$ is the quotient of $\M(\a',\b')$, where $(\a',\b')$ is the maximum element of $\Betti(H)_{\le d}$, it follows that $\VB(H)_{\le d}$ is irreducible and unirational, and so is $\VB(H)$.
\end{proof}

\bigskip
The main result of this subsection is the following.

\begin{theorem}\label{Main}
The closed strata $\overline{\VB(\a,\b)}$ in $\VB(H)$ form a graded lattice dual to $\Betti(H)$ under the partial order of inclusion.
Furthermore, the intersection of two closed strata $\overline{\VB(\a,\b)}$ and $\overline{\VB(\a',\b')}$ is again a closed stratum $\overline{\VB(\a'',\b'')}$, where $(\a'',\b'')$ is the join of $(\a,\b)$ and $(\a',\b')$ in the lattice $\Betti(H)$.
\end{theorem}

The theorem needs several standard lemmas on the behavior of resolutions in families with constant Hilbert functions.
We include proofs here for the lack of appropriate references.

\begin{lemma}\label{Limit}
Let $\E' \in \VB(\a',\b')$ and suppose $(\a,\b)\preceq (\a',\b')$. 
Then there is a family of bundles $\E$ on $\P^n$ over a dense open set $U\subset \A^1$ containing the origin $0\in \A^1$, such that $\E_0\cong \E'$ and $\E_t\in \VB(\a,\b)$ for any closed point $0\ne t\in U$.
\end{lemma}
\begin{proof}
Suppose $(\a',\b') = (\a,\b)+\c$. 
By \Cref{Generalization}, the pair $(\a,\b)$ is admissible.
Let $\psi \in \M^0(\a',\b')$ be a minimal presentation of $\E'$, and let $\varphi \in \M^0(\a,\b)$ be a minimal presentation of a bundle $\E$. 
Set $\varphi' = \varphi\oplus \op{Id}_{\L(\c)}$ and consider the morphism $\Phi: \L(\b')\times \A^1 \to \L(\a')\times \A^1$ whose fiber over a closed point $t\in \A^1$ is given by $\Phi_t := \psi+t\cdot \varphi'$.
By \Cref{DepthOpen}, the morphism $\Phi_t \in \M(\a',\b')$ for all closed points $t$ in an open dense set $U\subset \A^1$ containing $0$.
This shows that $\coker \Phi_t \in \VB(\a,\b)^{\preceq}$ for $t\in U$.
We show that in fact $\coker \Phi_t \in \VB(\a,\b)$ for all $0\ne t\in U$.
Let $t\ne 0$ be any closed point of $U$. 
Since $\psi$ is minimal and $\varphi'$ induces an isomorphism on the common summand $\L(\c) \xrightarrow{\sim} \L(\c)$, it follows that $\Phi_t$ also splits off the common summand $\L(\c) \xrightarrow{t} \L(\c)$. 
Since $\varphi$ does not split off any common summands other than those of $\L(\c)$, neither does $\Phi_t$ by Nakayama's lemma.
It follows that the free resolution 
\[
0 \to L(\a') \xrightarrow{\Phi_t} L(\b') \to H^0(\E_t) \to 0
\]
contains a minimal one of the form
\[
0 \to L(\a) \to L(\b) \to H^0_*(\E_t) \to 0. \qedhere
\]
\end{proof}

\begin{lemma}\label{Limit2}
Let $\E$ be a family of bundles on $\P^n$ satisfying (\ref{Cohomology}) parametrized by a variety $T$, such that all fibers have the same Hilbert function $H$. 
Then general fibers have the same Betti numbers $(\a,\b)$, where $(\a,\b)\preceq (\a',\b')$ for the Betti numbers $(\a',\b')$ of any fiber $\E_t$.
\end{lemma}
\begin{proof}
Let $t\in T$ be a closed point.
We may base change to $\Spec \O_{T,t}$ and reduce to the case where $T$ is an affine local domain.  
Let $m$ be the maximal ideal of $T$ with residue field $k$ and set $R_T := T[x,y,z]$ and $R := k[x,y,z]$. 
The module $E:= \bigoplus_{l\in\mathbb{Z}} H^0(\E(l))$ is finitely generated over $R_T$ since $\E$ is a bundle. 
Since the fibers over $T$ have the same Hilbert functions, it follows that $E$ is flat over $T$.
If $\bigoplus_{i = 1}^{l+r} R(-b'_i) \xrightarrow{d} E\otimes_T k$ is a minimal system of generators, then by Nakayama's lemma over generalized local rings, it lifts to a system of generators
$\bigoplus_{i = 1}^{l+r} R_T(-b'_i) \xrightarrow{d_T} E$. 
Since $E$ is flat over $T$, so is $\ker d_T$ and thus $(\ker d_T) \otimes_T k \cong \ker d$.
Applying this procedure again, we find a free resolution of $E$
\[
F_\bullet:\quad 0 \to \bigoplus_{i = 1}^l R_T(-a'_i) \to \bigoplus_{i = 1}^{l+r} R_T(-b'_i) \to E \to 0
\]
that specializes to a minimal free resolution of $E\otimes_T k$.
It follows that $F_\bullet \otimes_T k(T)$ is a free resolution of the generic fiber which contains a minimal free resolution of the form 
\[
0\to \bigoplus_{i = 1}^j R_T(-a_i)\otimes_T k(T) \to \bigoplus_{i = 1}^{j+r} R_T(-b_i)\otimes_T k(T)\to E\otimes_T k(T) \to 0.
\]
We conclude that the general fibers $\E_t$ have the Betti numbers $(\a,\b)\preceq (\a',\b')$.
\end{proof}

\begin{lemma}\label{Dual}
For $(\a,\b),(\a',\b')\in \Betti(H)$ the following are equivalent.
\begin{enumerate}
\item $(\a,\b) \preceq (\a',\b')$,
\item $\overline{\VB(\a,\b)} \supseteq \VB(\a',\b')$,
\item $\VB(\a',\b')\cap \overline{\VB(\a,\b)} \ne \varnothing$,
\item $\VB(\a,\b)\subseteq \VB(\a',\b')^{\preceq}$. 
\end{enumerate}
Here all closures are taken within $\VB(H)$.
\end{lemma}
\begin{proof}
(1) $\Longrightarrow$ (2): Suppose $(\a',\b') =(\a,\b)+\c$.
Let $\varphi \in \M^0(\a,\b)$ and $\psi \in \M^0(\a',\b')$.
Consider the line $\Phi: \A^1 \hookrightarrow \A(\a',\b')$ defined $\Phi(t) := \psi+ t\cdot \varphi'$, where $\varphi' =  \varphi \oplus \op{Id}_{\L(\c)}$.
For an open set $U\subset \A^1$ containing $0$, the image $\Phi(t)$ is contained in $\M(\a',\b')$. 
By \Cref{Limit}, the image of $\Phi(t)$ in the quotient $\VB(\a',\b')^{\preceq}$ lies in $\VB(\a,\b)$ for $t\ne 0$.
It follows that the image of $\psi$ in $\VB(\a',\b')$ is contained in the closure of $\VB(\a,\b)$ inside the space $\VB(\a',\b')^{\preceq}$.
Since $\psi$ represents an arbitrary point of $\VB(\a',\b')$, we conclude that $\VB(\a',\b')$ is contained in the closure of $\VB(\a,\b)$ in $\VB(\a',\b')^{\preceq}$,
and therefore the same is true inside $\VB(H)$. 

(2) $\Longrightarrow$ (3) is trivial.

(1) $\Longrightarrow$ (4) is proven in \Cref{Subspace}.

(3) $\Longrightarrow$ (1): Let $d := \max(\reg (\a,\b),\reg(\a',\b'))$.
Let $(\a'',\b'')$ denote the maximum element of $\Betti(H)_{\le d}$. 
Let $\pi:\M(\a'',\b'') \to \VB(H)_{\le d}$ be the quotient map and set $V$ to be the preimage of $\overline{\VB(\a,\b)}$ under $\pi$, endowed with the structure of a (reduced) subvariety of $\M(\a'',\b'')$.
Let $\E$ be the pullback of the tautological family of bundles $\E(\a'',\b'')$ on $\M(\a'',\b'')$ to $V$. 
Since $\VB(\a,\b)$ is dense in $\overline{\VB(\a,\b)}$, it follows that the fiber $\E_v$ over a general point $v\in V$ has Betti numbers $(\a,\b)$.
If $p$ is a point in $\VB(\a',\b')$ that is in the closure of $\VB(\a,\b)$, and $q$ is a point in $\pi^{-1}(p)$, then $q\in V$ and $\E_q$ has Betti numbers $(\a',\b')$.  
Finally, an application of \Cref{Limit2} to the family $\E$ gives $(\a,\b)\preceq (\a',\b')$.

(4) $\Longrightarrow$ (1): If $\E$ is a bundle with a free resolution of the form 
\[
0 \to L(\b') \to L(\a') \to H^0_*(E) \to 0,
\]
 then it contains as a summand the minimal free resolution of $\E$ 
\[
0 \to L(\b) \to L(\a) \to H^0_*(E) \to 0
\] 
with a direct complement of the form
\[
0 \to L(\c) \xrightarrow{\sim} L(\c) \to 0
\]
for some $(\a',\b') = (\a,\b)+\c$. 
It follows that $(\a,\b)\preceq (\a',\b')$.
\end{proof}

\begin{proof}[Proof of \Cref{Main}]
The first statement follows directly from \Cref{Dual}.
For the same reason, it is clear that $\overline{\VB(\a'',\b'')}$ is in the intersection of $\overline{\VB(\a,\b)}$ and $\overline{\VB(\a',\b')}$.
Let $p$ be a closed point in the intersection of $\overline{\VB(\a,\b)}$ and $\overline{\VB(\a',\b')}$.
We assume $p\in \VB(\c,\d)$ for some $(\c,\d)\in \Betti(H)$ since $\M(H)$ is the disjoint union of these subspaces. 
By \Cref{Dual}, it follows that $(\a,\b)\preceq (\c,\d)$ and $(\a',\b')\preceq(\c,\d)$. 
Since $(\a'',\b'') \preceq (\c,\d)$ by the definition of join, another application of \Cref{Dual} shows that $p\in \overline{\VB(\a'',\b'')}$.
\end{proof}

\bigskip

Last but not least, we discuss the semistable case where the description of the stratification holds within the coarse moduli space. 

\bigskip

By \cite[Theorem 4.2]{Maruyama2}, semistablity is open for a family of torsion-free sheaves. 
Furthermore, the set of semistable torsion sheaves with a given Hilbert polynomial $\chi$ is bounded in the sense of Maruyama, and thus have bounded regularity by \cite[Theorem 3.11]{Maruyama2}. 
Let $\VB(H)^{ss}$ and $\VB(\a,\b)^{ss}$ denote the subset of isomorphism classes of semistable bundles in $\VB(H)$ and $\VB(\a,\b)$ respectively.
It follows that $\VB(H)^{ss}$ and all $\VB(\a,\b)^{ss}$ are contained in $\VB(H)_{\le d}$ for some large enough integer $d$. 
Since $\VB(\a,\b)^{ss}$ is open in $\VB(\a,\b)$ and $\VB(H)^{ss}$ is open in $\VB(H)$ by the similar reasoning as in \Cref{Dense}, it follows that the stratification of $\VB(H)^{ss}$ by $\VB(\a,\b)^{ss}$ has the same description as given in \Cref{Main}.

Let $\M(\chi)$ denote the coarse moduli space of semistable sheaves on $\P^n$ with Hilbert polynomial $\chi$.
We show that the spaces $\VB(H)^{ss}$ and $\VB(\a,\b)^{ss}$ are subschemes of $\M(\chi)$.
Let $\M^0(\a,\b)^{ss}$ denote the open subscheme of $\M^0(\a,\b)$ over which the fibers of the tautological family of bundles $\E^0(\a,\b)$ are semistable.
By the property of the coarse moduli space, there is a map $p_0:\M^0(\a,\b)^{ss} \to \mathcal{M}(\chi)$ inducing the family of semistable bundles.
By \Cref{Orbit}, the isomorphism classes of the fibers are exactly given by the $G(\a,\b)$-orbits.
Therefore $\VB(\a,\b)^{ss}$ is a subscheme of $\M(\chi)$ with the image subscheme of $p$.
Similarly, the space $\VB(H)_{\le d}$ is also a subscheme of $\M(\chi)$. 
Since $\VB(H)^{ss}$ is an open subspace of $\VB(H)_{\le d}$ for some $d\gg 0$, the same is true for $\VB(H)^{ss}$.


\begin{thebibliography}{9}
\bibitem{Atiyah}
Atiyah, M.
\emph{On the Krull-Schmidt theorem with application to sheaves}. 
Bull. Soc. Math. France 84 (1956), 307–317. 

\bibitem{Barth}
Barth, W.
\emph{Moduli of vector bundles on the projective plane}. 
Invent. Math. 42 (1977), 63–91.

\bibitem{BS}
Bohnhorst, G; Spindler, H.
\emph{The stability of certain vector bundles on $\P^n$}. 
Complex algebraic varieties (Bayreuth, 1990), 39–50, 
Lecture Notes in Math., 1507, Springer, Berlin, 1992.




\bibitem{DM} 
Dionisi, C; Maggesi, M.
\emph{Minimal resolution of general stable rank-2 vector bundles on $\P^2$}.
Boll. Unione Mat. Ital. Sez. B Artic. Ric. Mat. (8) 6 (2003), no. 1, 151–160. 

\bibitem{Potier}
Drezet, J.-M.(F-PARIS7); Le Potier, J.(F-PARIS7).
\emph{Fibr\'{e}s stables et fibr\'{e}s exceptionnels sur $\P^2$}. (French. English summary) [Stable bundles and exceptional bundles on $\P^2$]
Ann. Sci. \'{E}cole Norm. Sup. (4) 18 (1985), no. 2, 193–243.

\bibitem{EN}
Eagon, J. A.; Northcott, D. G.
\emph{Ideals defined by matrices and a certain complex associated with them}. 
Proc. Roy. Soc. Ser. A 269 1962 188–204. 

\bibitem{DE}
Eisenbud, D.
\emph{Commutative algebra. With a view toward algebraic geometry}. 
Graduate Texts in Mathematics, 150. Springer-Verlag, New York, 1995. xvi+785 pp. 

\bibitem{ES}
Ellingsrud, G.(N-OSLO); Strømme, S. A.(N-BERG)
\emph{On the rationality of the moduli space for stable rank-2 vector bundles on $\P^2$}. Singularities, representation of algebras, and vector bundles (Lambrecht, 1985), 363–371,
Lecture Notes in Math., 1273, Springer, Berlin, 1987.

\bibitem{EG}
Evans, E. Graham; Griffith, Phillip.
\emph{The syzygy problem}.
Ann. of Math. (2) 114 (1981), no. 2, 323–333.

\bibitem{GM}
Giusti, M; Merle, M.
\emph{Singularit\'{e}s isol\'{e}es et sections planes de vari\'{e}t\'{e}s d\'{e}terminantielles. II. Sections de vari\'{e}t\'{e}s d\'{e}terminantielles par les plans de coordonn\'{e}es}. Algebraic geometry (La Rábida, 1981), 103–118, 
Lecture Notes in Math., 961, Springer, Berlin, 1982. 

\bibitem{Horrocks}
Horrocks, G.; Mumford, D.
\emph{A rank 2 vector bundle on $\P^4$ with 15,000 symmetries}.
Topology 12 (1973), 63–81.

\bibitem{M2}
Grayson, D; Stillman, M.
\emph{Macaulay2, a software system for research in algebraic geometry},
Available at \url{http://www.math.uiuc.edu/Macaulay2/}.

\bibitem{EGA}
Grothendieck, A.
\emph{\'{E}l\'{e}ments de g\'{e}om\'{e}trie alg\'{e}brique : IV. \'{E}tude locale des sch\'{e}mas
et des morphismes de sch\'{e}mas, Quatri\'{e}me partie}.
Publications math\'{e}matiques de l’I.H.\'{E}.S., tome 32 (1967), p. 5-361

%

\bibitem{Hartshorne2}
Hartshorne, R.
\emph{Stable vector bundles of rank 2 on $\P^3$}. 
Math. Ann. 238 (1978), no. 3, 229–280. 

\bibitem{Hartshorne}
Hartshorne, Robin.
\emph{Varieties of small codimension in projective space}.
Bull. Amer. Math. Soc. 80 (1974), 1017–1032.

\bibitem{Problem}
Hartshorne, Robin.
\emph{Algebraic vector bundles on projective spaces: a problem list}.
Topology 18 (1979), no. 2, 117–128.


\bibitem{Hulek}
Hulek, K.
\emph{Stable rank-2 vector bundles on $\P^2$ with $c_1$ odd}. 
Math. Ann. 242 (1979), no. 3, 241–266. 



\bibitem{Maruyama1}
Maruyama, Masaki.
\emph{Stable vector bundles on an algebraic surface}.
Nagoya Math. J. 58 (1975), 25–68.

\bibitem{Maruyama}
Maruyama, M.
\emph{Moduli of stable sheaves. II}. 
J. Math. Kyoto Univ. 18 (1978), no. 3, 557–614.

\bibitem{Maruyama2}
Maruyama, M.
\emph{The rationality of the moduli spaces of vector bundles of rank 2 on $\P^2$}. 
With an appendix by Isao Naruki. Adv. Stud. Pure Math., 10, Algebraic geometry, Sendai, 1985, 399–414, North-Holland, Amsterdam, 1987.

\bibitem{Gap}
Maruyama, Masaki(J-KYOT).
\emph{The rationality of the moduli spaces of vector bundles of rank 2 on $\P^2$}.
With an appendix by Isao Naruki. Adv. Stud. Pure Math., 10, Algebraic geometry, Sendai, 1985, 399–414, North-Holland, Amsterdam, 1987.



%


\bibitem{Serre}
Serre, J.P.
\emph{Alg\`{e}bre locale. Multiplicit\'{e}s}. (French) 
Cours au Coll\`{e}ge de France, 1957–1958, r\'{e}dig\'{e} par Pierre Gabriel. Seconde \'{e}dition, 1965. Lecture Notes in Mathematics, 11 Springer-Verlag, Berlin-New York 1965 vii+188 pp.

\bibitem{MZ}
Zhang, M.
\texttt{BundlesOnPn}, \emph{a Macaulay2 Package}.
Available at \url{https://math.berkeley.edu/~myzhang/Macaulay2/BundlesOnPn.m2}.
\end{thebibliography}
\end{document}